\documentclass[a4paper,10pt]{article}
\usepackage{mathtext}
\usepackage[T1,T2A]{fontenc}
\usepackage[cp1251]{inputenc}
\usepackage[english]{babel}
\usepackage{amsmath}
\usepackage{amsfonts}
\usepackage{amssymb}
\usepackage{mathrsfs}
\usepackage{amsthm}
\usepackage{enumerate}
\usepackage{graphicx}
\usepackage{wrapfig}

\usepackage{color}
\usepackage{euscript}

\usepackage{cite}

\newtheorem{Le}{Lemma}[section]
\newtheorem{Def}[Le]{Definition}

\newtheorem{Th}{Theorem}[section]
\newtheorem{Cor}[Le]{Corollary}
\newtheorem{Rem}[Le]{Remark}
\newtheorem{Conj}[Le]{Conjecture}

\newtheorem{Ex}[Le]{Example}
\numberwithin{equation}{section}

\newcommand{\R}{\mathbb{R}}

\newcommand{\N}{\mathbb{N}}
\newcommand{\Z}{\mathbb{Z}}

\newcommand{\eps}{\varepsilon}

\newcommand{\eq}[1]{\begin{equation}{#1}\end{equation}}
\newcommand{\mlt}[1]{\begin{multline}{#1}\end{multline}}
\newcommand{\alg}[1]{\begin{align}{#1}\end{align}}

\newcommand{\set}[2]{\{{#1}\mid{#2}\}}

\newcommand{\Lseqref}[1]{\stackrel{\scriptscriptstyle{\eqref{#1}}}{\lesssim}}

\newcommand{\Lsref}[1]{\stackrel{#1}{\lesssim}}

\DeclareMathOperator{\supp}{supp}

\newcommand{\E}{\mathbb{E}}

\newcommand{\Mp}{\mathcal{M}_p}

\newcommand{\D}{\mathcal{D}}

\title{Fractional integration of summable functions: Maz'ya's~$\Phi$-inequalities}
\author{Dmitriy Stolyarov\thanks{Supported by Russian Science Foundation grant N 19-71-10023}}
\begin{document}
\maketitle
\begin{abstract}
We study the inequalities of the type $|\int_{\R^d} \Phi(K*f)| \lesssim \|f\|_{L_1(\R^d)}^p$, where the kernel $K$ is homogeneous of order $\alpha - d$ and possibly vector-valued, the function $\Phi$ is positively $p$-homogeneous, and $p = d/(d-\alpha)$. Under mild regularity assumptions on $K$ and $\Phi$, we find necessary and sufficient conditions on these functions under which the inequality holds true  with a uniform constant for all sufficiently regular functions $f$. 
\end{abstract}

\bigskip
\section{Hardy--Littlewood--Sobolev inequality and its modifications at the endpoint}
The classical Hardy--Littlewood--Sobolev inequality states that for any~$p,q$ such that~$1 < p < q < \infty$ and~$1/p - 1/q = \alpha/d$ there exists a constant~$C$ such that
\eq{\label{OriginalHLS}
\big\||\cdot|^{\alpha - d}*f\big\|_{L_q(\R^d)} \leq C\|f\|_{L_p(\R^d)}
}
for any~$f\in L_p(\R^d)$. The inequality was invented by Sobolev in~\cite{Sobolev1938} as an instrument to prove what is now called the Sobolev embedding theorem. The inequality fails when~$p=1$ as one may see by plugging an approximate identity that mimics a delta-measure for~$f$. 

In recent years there is a constant interest in corrections of the Hardy--Littlewood--Sobolev inequality at the endpoints. Such corrections may be obtained in different ways. A classical example follows from the Gagliardo--Nirenberg embedding~$\dot{W}_1^1(\R^d) \hookrightarrow L_{\frac{d}{d-1}}(\R^d)$,~$d > 1$:
\eq{\label{GN}
\big\||\cdot|^{1 - d}* f\big\|_{L_{\frac{d}{d-1}}(\R^d)} \lesssim \|f\|_{L_1(\R^d)},\quad f= \nabla g,
}
provided~$g$ is a smooth compactly supported function. The convolution is applied to a vector-valued function~$\nabla g$ coordinate-wise, and the way we measure the~$L_p$-norm of a vectorial function is not important. Here and in what follows the notation~$\lesssim$ signifies certain uniformity of the constant that is clear from the context.  The inequalities of this type, with some additional conditions on~$f$, are usually called Bourgain--Brezis inequalities. We refer the reader to the papers~\cite{BourgainBrezis2002, BourgainBrezis2004, BourgainBrezis2007, BousquetVanSchaftingen2014, HernandezSpector2020, KMS2015, LanzaniStein2005, Mazya2007, Mazya2010, Raita2018, Raita2019, SpectorVanSchaftingen2019, Spector2020, Stolyarov2020, Stolyarov2021, VanSchaftingen2004one, VanSchaftingen2004, VanSchaftingen2008, VanSchaftingen2010, VanSchaftingen2013} among many others and to the surveys~\cite{VanSchaftingen2014, Spector2019} for information on Bourgain--Brezis inequalities. The main heuristic principle beyond these inequalities is that a modification of the Hardy--Littlewood--Sobolev inequality for~$p=1$ is valid provided one cannot plug a delta-measure into it. In the example above,~$\nabla g$ cannot be a vectorial delta measure no matter how bad the distribution~$g$ is. 

In~\cite{Mazya2010}, Vladimir Maz'ya suggested a modification of the Hardy--Littlewood--Sobolev inequality that allows the substitution of a delta measure. Namely, he conjectured that the inequality
\eq{\label{MazjaInequality}
\Big|\int\limits_{\R^d}\Phi(\nabla f(x))\,dx\Big| \lesssim \|\Delta f\|_{L_1(\R^d)}^{\frac{d}{d-1}},\qquad f\in C_0^\infty(\R^d),
}
holds true whenever~$\Phi \colon \R^d \to \R$ is a locally Lipschitz positively~$\frac{d}{d-1}$-homogeneous function satisfying
\eq{\label{MazjaCancellation}
\int\limits_{S^{d-1}} \Phi(\zeta)\,d\sigma(\zeta) = 0.
}
In the latter condition,~$\sigma$ stands for the~$(d-1)$-dimensional Hausdorff measure on the unit sphere~$S^{d-1}$. By the positive~$p$-homogeneity of a function~$\Psi\colon \R^d \to \R$ we mean the validity of the identity~$\Psi(t x) = t^p \Psi(x)$ for any~$x\in \R^d$ and any~$t > 0$. The necessity of condition~\eqref{MazjaCancellation} for~\eqref{MazjaInequality} may be verified, as it usually happens in this field, by the example of~$f$ such that~$\Delta f$ mimics a delta measure. The conjecture was also listed as Problem~$5.1$ in~\cite{Mazya2018}. Some particular cases had been considered in~\cite{MazyaShaposhnikova2009}.

The main purpose of this paper is to prove Maz'ya's conjecture and provide a more systematic study of such type inequalities, which we will call Maz'ya's~$\Phi$-inequalities. The methods are classical and the paper is self-contained with a small exception for the Three Lattice Theorem, which we will use in Section~\ref{S3} below. The appearance of this theorem is not surprising since it is commonly used to transfer martingale inequalities to the Euclidean setting (e.g., see~\cite{Hytonen2014, HLP2013}). The probabilistic versions of Maz'ya's~$\Phi$-inequalities were obtained in~\cite{Stolyarov2021bis} (see~\cite{ASW2018} and~\cite{Stolyarov2019} for probabilistic versions of more classical Bourgain--Brezis inequalities), so we only need to transfer them to the Euclidean setting. Since the problem is non-linear, such a transference is non-trivial. In fact, we will not transfer the inequality, but rather re-prove it following the plot of the probabilistic proof and finding good representatives for its main characters. See Section~\ref{SD} for more explanation for the correspondence between the martingale and Euclidean settings.

I express my gratitude to Vladimir Maz'ya for bringing his conjecture to my attention and to Ilya Zlotnikov for reading the paper and improving the presentation.

\section{General Maz'ya's $\Phi$-inequalities}
Let~$d$ and~$\ell$ be natural numbers. Let~$\tilde{K}\colon S^{d-1}\to \R^\ell$ be a Lipschitz function. Let~$\alpha \in (0,d)$. Consider the kernel~$K\colon \R^d\to\R^\ell$ that is homogeneous of degree~$\alpha - d$
\eq{
K(x) = |x|^{\alpha - d}\tilde{K}(x/|x|),\qquad x\in \R^d\setminus \{0\}.
} 
The kernel~$K$ depends on~$\alpha$, however, we will suppress this dependence in our notation. Let~$p = d/ (d-\alpha)$ and let~$\Phi\colon \R^\ell \to \R$ be a positively $p$-homogeneous locally Lipschitz function that satisfies the cancellation conditions
\eq{\label{Cancellation}
\int\limits_{S^{d-1}} \Phi(\tilde{K}(\zeta))\,d\sigma(\zeta) = 0,\qquad \int\limits_{S^{d-1}} \Phi(-\tilde{K}(\zeta))\,d\sigma(\zeta) = 0.
}
Note that the two conditions are, in general, independent.
\begin{Th}\label{Main}
Let~$p\in(1,\infty)$. The inequality
\eq{\label{MI}
\Big|\int\limits_{\R^d} \Phi\big(K*f(x)\big)\,dx\Big| \lesssim \|f\|_{L_1(\R^d)}^p
}
holds true with a uniform constant for all functions~$f\in C_0^\infty(\R^d)$ with zero integral.
\end{Th}
Note that the particular case~$\tilde{K}(\zeta) = \zeta$,~$\alpha = 1$, and~$p = \frac{d}{d-1}$ of the theorem implies Maz'ya's original conjecture since
\eq{\label{Lapl}
\nabla f (x) = c_d \int\limits_{\R^d}\frac{y}{|y|^d}\Delta f(x-y)\,dy
} 
and for the kernel~$\tilde{K}(\zeta) = \zeta$ the two conditions~\eqref{Cancellation} reduce to~\eqref{MazjaCancellation}. One may find versions of Theorem~\ref{Main} generated by other (possibly, vectorial) elliptic homogeneous operators similar as~\eqref{Lapl} is generated by the Laplacian. 

We will prove Theorem~\ref{Main} in the case~$p \leq 2$ and will provide hints to the proof in the (easier) case~$p > 2$ in Section~\ref{SD} below. The cancellation conditions~\eqref{Cancellation} are necessary for~\eqref{MI}, as may be seen by plugging~$f$ that mimics a delta measure. 
\begin{Rem}
The assumption~$f\in C_0^\infty$ in Theorem~\ref{Main} is redundant. In the classical linear translation invariant inequalities \textup(such as~\eqref{OriginalHLS} or~\eqref{GN}\textup) one may pass from the case of regular functions to the case of general measures by standard approximation\textup, once the constants in the inequality for regular functions are uniform. In our case\textup, some regularity assumption is needed\textup: if we plug~$f = \delta_0 - \delta_z$ for some~$z\ne 0$\textup, the integral on the left hand side is undefined as a Lebesgue integral\textup; thus\textup, we cannot formulate Theorem~\ref{Main} in the case when~$f$ is an arbitrary measure of bounded variation. Any reasonable regularization makes the left hand side vanish. The condition that~$f$ has zero integral is needed to avoid a similar effect at infinity. 
\end{Rem}
Before we pass to the plan of the proof, we give two interesting examples.
\begin{Ex}
Let~$d=1$\textup,~$\ell = 1$\textup,~$\alpha = \frac12$\textup, and~$p=2$. Let also the kernel~$\tilde{K} \colon \{-1,1\}\to \R$ be given by the rule
\eq{
K(\zeta) = \zeta,\quad \zeta = \pm 1.
}
In other words\textup, we are interested in the inequality
\eq{
\Big|\int\limits_{\R}\Phi\Big(\int\limits_{\R}f(x-y)\frac{y}{|y|^{\frac32}}\,dy\Big)\,dx\Big| \lesssim \|f\|_{L_1(\R)}^2
}
for sufficiently regular functions~$f$ with zero integral. The cancelation conditions~\eqref{Cancellation} are then reduced to
\eq{
\Phi(1) + \Phi(-1) = 0,
}
which\textup, together with the condition that~$\Phi$ is positively~$2$-homogeneous\textup, implies
\eq{
\Phi(t) = t|t|,\qquad t\in\R.
}
Theorem~\ref{Main} yields the inequality
\eq{
\Big|\int\limits_{\R}\Big(\int\limits_{\R}f(x-y)\frac{y}{|y|^{\frac32}}\,dy\Big)\Big|\int\limits_{\R}f(x-y)\frac{y}{|y|^{\frac32}}\,dy\Big|\,dx\Big| \lesssim \|f\|_{L_1(\R)}^2
}
Note that for the most classical kernel~$\tilde{K}(\zeta) = |\zeta|$\textup, Theorem~\ref{Main} does not provide any interesting information since~\eqref{Cancellation} implies~$\Phi = 0$.
\end{Ex}
\begin{Ex}
Let now~$d=2$\textup,~$\ell=2$\textup,~$\alpha= 1$\textup, and~$p=2$. Let~$\tilde{K}(\zeta) = \zeta$ and let~$\Phi$ be a quadratic function
\eq{
\Phi(x) = a_{11}x^2 + a_{12}x_1x_2 + a_{22}x^2,\qquad x\in \R^2.
}
Theorem~\ref{Main} states \textup(in view of~\eqref{Lapl}\textup)  the inequality
\eq{
\Big|\int\limits_{\R^2}\Big(a_{11}\Big|\frac{\partial f}{\partial x_1}\Big|^2 +a_{12}\frac{\partial f}{\partial x_1}\frac{\partial f}{\partial x_2}+ a_{22}\Big|\frac{\partial f}{\partial x_2}\Big|^2\Big)\Big| \lesssim \|\Delta f\|_{L_1}^2
}
holds true if and only if~$a_{11} + a_{22}=0$. This inequality had been obtained in~\textup{\cite{MazyaShaposhnikova2009}} by classical Fourier-analytic methods.
\end{Ex}
Now, we present the plan of the proof. Let us first split the kernel into similar parts:
\eq{\label{Kn}
K_n(x) = \begin{cases}
K(x),\qquad &2^{-n-1} \leq |x| \leq 2^{-n};\\
0,\qquad &\text{otherwise}.
\end{cases}
}
Then~$K(x) = \sum_{n\in \Z} K_n(x)$ and 
\eq{
K_n(x) = 2^{(d-\alpha)n} K_0(2^n x).
}
We will also use the notation
\eq{
K_{\leq n} = \sum\limits_{k \leq n} K_k.
}
It is important that the kernels~$K_n$ have disjoint (up to a set of measure zero) supports and satisfy the cancellation condition similar to~\eqref{Cancellation}. In particular,
\eq{\label{IndCanc}
\int\limits_{\R^d}\Phi(\lambda K_n(x-y))\,dx = 0,\quad \forall y\in \R^d, n\in \Z, \lambda \in \R.
}

Our first target is to split Theorem~\ref{Main} into several simpler statements. The proofs of these statements will be presented in the following sections. 

By dilation invariance of the problem, we may assume~$\supp f \subset B_{\frac12}(0)$ in the proof of Theorem~\ref{Main}. By~$B_r(x)$ we always mean the Euclidean ball with center~$x\in \R^d$ and radius~$r > 0$.
\begin{Le}\label{FirstLemma}
Let~$p\in(1,\infty)$. If~$\supp f \subset B_{\frac12}(0)$ and~$\int f = 0$\textup, then
\eq{
\int\limits_{\R^d} \Big| K_{\leq 0}*f(x)\Big|^p\,dx \lesssim \|f\|_{L_1(\R^d)}^p,\qquad f\in L_1(\R^d).
}
\end{Le}
Unless otherwise stated, the absolute value of a vector denotes its Euclidean norm. 
\begin{Def}
Let~$p \in (1,2]$. Define the function~$\Mp\colon \R_+\times \R_+ \to \R$ by the rule
\eq{
\Mp(x,y) = \min(x^{p-1} y, x y^{p-1}).
}
\end{Def}
The function~$\Mp$ is pivotal for our considerations. It was suggested by the martingale analogs of Maz'ya's conjecture in~\cite{Stolyarov2021bis}. The most important feature that distinguishes if from more natural positively~$p$-homogeneous functions~$(y,z)\mapsto yz^{p-1}$ and~$(y,z)\mapsto y^{p-1}z$ is the local Lipschitz property (see Lemma~\ref{MpLip} below). In the case~$p=2$ the function~$\Mp$ simplifies to~$(y,z)\mapsto yz$.
\begin{Le}\label{SecondLemma}
Let~$p \leq 2$. For any~$n \geq 0$ and any~$f\in C_0^\infty(\R^d)$\textup, we have
\mlt{\label{SecondLemmaIneq}
\Big|\int\limits_{\R^d}\Big(\Phi\big(K_{\leq n+1}*f\big) - \Phi\big(K_{\leq n}*f\big)\Big)\Big|\\ \lesssim \Big|\int\limits_{\R^d}\Phi\big(K_{n+1}*f\big)\Big| + \int\limits_{\R^d}\Mp\Big(\big|K_{\leq n}*f(x)\big|, \big|K_{n+1}*f(x)\big|\Big)\,dx.
}
The constant in this inequality depends neither on~$f$ nor on~$n$.
\end{Le}

If we treat Theorem~\ref{Main} as a substitution for the~$L_1\to L_p$ continuity of the Riesz potential, the next theorem might be thought of as a substitution for the~$L_1 \to B_{p}^{0,1}$ continuity, where~$B_p^{0,1}$ is the Besov space (see~\cite{Stolyarov2020} for a similar reduction for more classical Bourgain--Brezis inequalities).
\begin{Th}\label{Main2}
For any~$f\in C_0^\infty(\R^d)$\textup,
\eq{
\sum\limits_{n \geq 1}\Big|\int\limits_{\R^d} \Phi(K_n*f)\Big| \lesssim \|f\|_{L_1}^p.
}
The constant in this inequality depends neither on~$f$ nor on~$n$.
\end{Th}
\begin{Th}\label{Remainder}
For any~$f\in C_0^\infty(\R^d)$\textup,
\eq{
\sum\limits_{n \geq 0} \int_{\R^d}\Mp\Big(\big|K_{\leq n}*f(x)\big|, \big|K_{n+1}*f(x)\big|\Big)\,dx \lesssim \|f\|_{L_1}^p.
}
The constant in this inequality depends neither on~$f$ nor on~$n$.
\end{Th}
We will later provide local (with respect to~$n$) inequalities, which imply Theorems~\ref{Main2} and~\ref{Remainder}. Now we will show how these theorems imply Theorem~\ref{Main}.
\begin{proof}[Derivation of Theorem~\ref{Main} from other theorems and lemmas in the case $p\leq 2$.]
Let~$\supp f \subset B_{\frac12}(0)$.  We estimate the left hand side  of~\eqref{MI} with the telescopic sum
\eq{
\Big|\int\limits_{\R^d}\Phi(K*f)\Big| \leq \Big|\int\limits_{\R^d}\Phi(K_{\leq 0}*f)\Big| + \sum\limits_{n \geq 0}\Big|\int\limits_{\R^d}\Phi(K_{\leq n+1}*f) - \int\limits_{\R^d}\Phi(K_{\leq n}*f)\Big|,
}
the series and the integrals converge since~$f\in C_0^\infty$; the only question is to obtain a uniform bound by~$\|f\|_{L_1}$. The first term is estimated with the help of Lemma~\ref{FirstLemma} and the inequality~$|\Phi(x)| \lesssim |x|^p$. Lemma~\ref{SecondLemma}, in its turn, reduces the bound of the second term to Theorems~\ref{Main2} and~\ref{Remainder}.
\end{proof}
Section~\ref{S2} contains the ``classical'' part of the proof that does not have the martingale flavor. This includes the proofs of Lemmas~\ref{FirstLemma} and~\ref{SecondLemma}, and the proofs of Theorems~\ref{Main2} and~\ref{Remainder} in the case~$p=2$ (in this case some parts of the argument simplify and we wish to present these simplifications). In particular, Section~\ref{S2} contains the proof of Theorem~\ref{Main} in the case~$p=2$. The case of general~$p$ is more complicated, it may be found in Section~\ref{S3} below. Section~\ref{SD} provides supplementary information such as comments on the case~$p\geq 2$, the correspondence between the martingale and the Euclidean settings, and suggestions for further study. The last Section~\ref{SAI} contains auxiliary inequalities that will be used in the proofs.

%

\section{Real variable techniques}\label{S2}
\begin{proof}[Proof of Lemma~\ref{FirstLemma}.]
The kernel~$K_{\leq 0}$ is uniformly bounded, therefore, it suffices to prove the estimate
\eq{\label{Le2}
\int\limits_{|x| \geq 2} \Big|K_{\leq 0}*f(x)\Big|^p\,dx \lesssim \|f\|_{L_1(\R^d)}^p.
}
We use the condition~$\int f = 0$ and Lemma~\ref{K1} below:
\mlt{
\big| K_{\leq 0}*f(x) \big| = \Big|\int\limits_{\R^d}  \big[K_{\leq 0}(x-y) - K_{\leq 0}(x)\big]f(y)\,dy\Big|\\ \Lsref{\text{\tiny Lem. }\scriptscriptstyle \ref{K1}} \int\limits_{\R^d}|f(y)|\frac{|y|}{|x|^{d-\alpha+1}}\,dy \lesssim \frac{\|f\|_{L_1}}{|x|^{d-\alpha + 1}},\quad |x| \geq 2,
}
raise this estimate to the power~$p$, integrate with respect to~$x$, and obtain~\eqref{Le2}.
\end{proof}

\begin{proof}[Proof of Lemma~\ref{SecondLemma}.]
Let
\eq{
a_n(x) = K_{\leq n}*f (x);\quad b_n(x) = K_{n+1}*f(x). 
}
Let also~$A_n = \set{x\in \R^d}{|a_n(x)| \geq |b_n(x)|}$ and let~$B_n = \R^d \setminus A_n$. From Lemma~\ref{Phi1} below, we deduce
\alg{
\Big|\Phi\Big(a_n(x) + b_n(x)\Big) - \Phi\big(a_n(x)\big)\Big| \lesssim |a_n(x)|^{p-1}|b_n(x)| = \Mp(|a_n(x)|,|b_n(x)|),\quad &x\in A_n;\\
\Phi\Big(a_n(x) + b_n(x)\Big) - \Phi\big(a_n(x)\big) = \Phi\big(b_n(x)\big) + O\Big(\Mp\big(|a_n(x)|, |b_n(x)|\big)\Big),\quad &x\in B_n.
}
Therefore,
\mlt{
\Big|\int\limits_{\R^d}\Big(\Phi\big(a_n(x) + b_n(x)\big) - \Phi\big(a_n(x)\big)\Big)\,dx\Big|\lesssim \int\limits_{\R^d}\Mp\big(|a_n(x)|, |b_n(x)|\big)\,dx + \Big|\int\limits_{B_n}\Phi\big(b_n(x)\big)\,dx\Big|\\ \leq 
\Big|\int\limits_{\R^d}\Phi\big(b_n(x)\big)\,dx\Big| + \Big|\int\limits_{A_n}\Phi\big(b_n(x)\big)\,dx\Big| + \int\limits_{\R^d}\Mp\big(|a_n(x)|, |b_n(x)|\big)\,dx\\ \lesssim \Big|\int\limits_{\R^d}\Phi\big(b_n(x)\big)\,dx\Big| + \int\limits_{\R^d}\Mp\big(|a_n(x)|, |b_n(x)|\big)\,dx.
}
If we return back to the original notation, we see that this is exactly~\eqref{SecondLemmaIneq}.
\end{proof}

\begin{proof}[Proof of Theorem~\ref{Remainder} in the case~$p=2$.]
In this case,~$\Mp(x,y) = xy$, and the things simplify. We start with reducing the quadratic bound to a linear one:
\mlt{
\sum\limits_{n \geq 0}\; \int\limits_{\R^d} |K_{\leq n}*f(x)||K_{n+1}*f(x)|\,dx \leq \int\limits_{\R^d}|f|\cdot\Big[\Big(\sum\limits_{n \geq  0} |K_{\leq n}|* |K_{n+1}|\Big)*|f|\Big]\\ \leq \Big\|\sum\limits_{n \geq  0} |K_{\leq n}|* |K_{n+1}|\Big\|_{L_\infty} \|f\|_{L_1}^2.
}
It remains to prove
\eq{\label{LinftyBound}
\Big\|\sum\limits_{n \geq  0} |K_{\leq n}|* |K_{n+1}|\Big\|_{L_\infty} \lesssim 1.
}
We state the local estimate (see Lemma~\ref{K2} below for the proof)
\eq{\label{LocalKernelBound}
|K_{\leq n}|* |K_{n+1}| (x) \lesssim 2^n|x|(1+2^n|x|)^{-\frac{d}{2}-1},
}
the constant is uniform with respect to~$x$ and~$n$. We note that the fact that~$K_{\leq n}$ and~$K_{n+1}$ have disjoint supports is crucial for this inequality. We fix arbitrary~$x\in \R^d$, and finish the proof of~\eqref{LinftyBound}:
\eq{\label{GeomSer}
\sum\limits_{n \geq  0} |K_{\leq n}|* |K_{n+1}|(x) \lesssim \sum\limits_{n\geq 0}2^n|x|(1+2^n|x|)^{-\frac{d}{2}-1} \leq\!\!\! \sum\limits_{n\colon 2^n|x|\leq 1}2^{n}|x| +\!\!\! \sum\limits_{n\colon 2^n|x|\geq 1}(2^{n}|x|)^{-\frac{d}{2}} \lesssim 1.
}
\end{proof}
The proof of Theorem~\ref{Main2} may also be simplified in the case~$p=2$, however, we prefer to argue for general~$p$ for some time. The following lemma is the core of the whole paper.
\begin{Le}\label{MedianLe}
Let~$f$ be a compactly supported summable function. Then\textup,
\eq{\label{MedianLeIneq}
\Big|\int\limits_{\R^d}\Phi(K_0*f(x))\,dx\Big| \lesssim \|f\|_{L_1}^{p-1}\inf_{c\in\R^d}\Big(\int\limits_{\R^d}|x-c||f(x)|\,dx\Big).
}
\end{Le}
\begin{proof}
Pick arbitrary~$y\in \R^d$. Using~\eqref{IndCanc} with~$n=0$, the inequality
 \eq{
 \Big|\Phi(a) - \Phi(b)\Big| \lesssim |a-b|\max(|a|,|b|)^{p-1},
 }
which follows from Lemma~\ref{Phi1} below, and the fact that~$K_0$ is a bounded function, we may write
\mlt{\label{LipschitzConvolution}
\Big|\int\limits_{\R^d}\Phi(K_0*f(x))\,dx\Big| = \Big|\int\limits_{\R^d}\Big(\Phi(K_0*f(x)) - \Phi\Big(K_0(x-y) \cdot \int f\Big)\Big)\,dx\Big|\\ \lesssim
\|f\|_{L_1}^{p-1}\int\limits_{\R^d}\Big|\int\limits_{\R^d} K_{0}(x-z)f(z)\,dz - \int\limits_{\R^d}K_0(x-y)f(z)\,dz\Big|\,dx\\ \leq
\|f\|_{L_1}^{p-1}\int\limits_{\R^d}\int\limits_{\R^d}|K_0(x-z) - K_0(x-y)|\,dx\; |f(z)|\,dz \lesssim\\ \|f\|_{L_1}^{p-1}\int\limits_{\R^{d}}|z-y||f(z)|\,dz.
}
The last inequality in the chain uses Lemma~\ref{K3} below. Since~$y$ is arbitrary, we get the desired inequality.
\end{proof}
Using dilation invariance, we may prove a similar inequality for any~$n$. An easy way to perform this computation is to consider dilations that preserve the~$L_1$ norm of~$f$ (i.e.~$f_n(x) = 2^{nd}f(2^n x)$) and notice that both the left hand and the right hand sides are preserved by such dilations (here one should use that~$p = d/(d-\alpha)$).
\begin{Cor}\label{MedianCor}
The inequality
\eq{\label{MedianCorIneq}
\Big|\int\limits_{\R^d}\Phi(K_n*f(x))\,dx\Big| \lesssim 2^{n}\|f\|_{L_1}^{p-1}\inf_{c\in\R^d}\int\limits_{\R^d}|x-c||f(x)|\,dx
}
holds true for any compactly supported summable function~$f$\textup; the constant is uniform with respect to~$n$.
\end{Cor}
The inequality~\eqref{MedianCorIneq} seems to be quite sharp when~$f$ is supported on a ball of radius~$\sim 2^{-n}$. In fact, Theorem~\ref{Main2} will be derived from Corollary~\ref{MedianCor} by splitting general~$f$ into parts and applying the corollary to each of the parts. 
\begin{Rem}
The quantity on the right hand side of~\eqref{MedianLeIneq} is monotone with respect to~$f$ in the following sense\textup: if~$|g| \leq |f|$ almost everywhere\textup, then
\eq{
\|g\|_{L_1}^{p-1}\inf_{c\in\R^d}\int\limits_{\R^d}|x-c||g(x)|\,dx \leq \|f\|_{L_1}^{p-1}\inf_{c\in\R^d}\int\limits_{\R^d}|x-c||f(x)|\,dx.
}
We will use this principle quite often without mention\textup; usually\textup, we will have~$g = f\chi_{\Omega}$ for some~$\Omega \subset \R^d$.
\end{Rem}

Let~$Q_{k,j}$ be the grid of dyadic cubes:
\eq{
Q_{k,j} = \prod\limits_{i=1}^d\Big[2^{-k}j_i,2^{-k}(j_i+1)\Big],\qquad k \geq 0,\quad j = (j_1,j_2,\ldots, j_d) \in \Z^d.
}
We are ready to state the local version of Theorem~\ref{Main2}. When~$Q$ is cube in~$\R^d$ whose sidelength is~$\ell(Q)$ and~$\lambda > 0$ is a scalar, the notation~$\lambda Q$ means the cube with the same center as~$Q$ and with the sidelength~$\lambda \ell(Q)$.
\begin{Th}\label{LocalMain2}
For any~$n \geq 0$, the inequality
\eq{\label{LocalMain2Ineq}
\Big|\int\limits_{\R^d}\Phi(K_{n+1}*f)\Big| \lesssim 2^n \sum\limits_{j\in\Z^d}\|f\|_{L_1(3^d Q_{n,j})}^{p-1}\inf\limits_{c_j}\int\limits_{3^d Q_{n,j}}|x-c_j||f(x)|\,dx
}
holds true with the constant independent of~$f$ and~$n$.
\end{Th}
\begin{proof}
By dilation invariance, it suffices to prove Theorem~\ref{LocalMain2} for~$n=0$ only. By Lemma~\ref{MedianLe}, we know the desired estimate for the case where~$f$ is supported on one of the cubes~$3^dQ_{n,j}$. We will be proving by induction with respect to~$D = 0,1,2,\ldots, d$ the following statement: \emph{the estimate~\eqref{LocalMain2Ineq} holds true when~$f$ is supported on a parallelepiped of dimensions
\eq{
\underbrace{\infty\times\infty\times\ldots\times\infty}_{D \text{\emph{ times}}}\times 3^{d-D}\times 3^{d-D}\times\ldots \times 3^{d-D}.
}  
}
The base follows from Lemma~\ref{MedianLe} (rather from Corollary~\ref{MedianCor} since we need to work with~$K_1$ instead of~$K_0$), the case~$D=d$ implies the theorem. Let us prove the induction step~$D \to D+1$. Let~$f$ be supported on~$\Pi$, where~$\Pi$ is the parallelepiped of dimensions 
\eq{
\underbrace{\infty\times\infty\times\ldots\times\infty}_{D+1 \text{ times}}\times 3^{d-D-1}\times 3^{d-D-1}\times\ldots \times 3^{d-D-1}.
}
We split it into the parallelepipeds~$\Pi_i$,~$i\in\Z$, of dimensions
\eq{\label{DimensionsOfPar}
\underbrace{\infty\times\infty\times\ldots\times\infty}_{D \text{ times}}\times 3^{d-D-1}\times 3^{d-D-1}\times\ldots \times 3^{d-D-1}.
}  
in a natural way (see Figure~\ref{SplittingIntoParal}).

Let~$f_i = f\chi_{\Pi_i}$. Note that the inductive hypothesis provides us with the inequality
\eq{\label{FromIndHyp}
\sum\limits_{i} \Big|\int\limits_{\R^d}\Phi(K_1*f_i)\Big| \lesssim \sum\limits_{j\in\Z^d}\|f\|_{L_1(3^d Q_{0,j})}^{p-1}\inf\limits_{c_j}\int\limits_{3^d Q_{0,j}}|x-c_j||f(x)|\,dx.
}
We will use that the functions~$K_1*f_i$ have almost disjoint supports. Let
\eq{
L_{i} = B_{\frac12}(\Pi_i)\cap B_\frac12(\Pi_{i+1}).
}
We may write:
\mlt{\label{LongFormula}
\Big|\int\limits_{\R^d}\Phi(K_1*f)\Big|\\
 \lesssim \sum\limits_{i}\Big|\int\limits_{\R^d} \Phi(K_1*{f_i})\Big| + \sum\limits_{i}\Big|\int\limits_{L_i}\Big(\Phi(K_1*(f_{i} + f_{i+1}))-\Phi(K_1*f_i)-\Phi(K_1*f_{i+1})\Big)\Big|\\ =
\sum\limits_{i}\Big|\int\limits_{\R^d} \Phi(K_1*{f_i})\Big| + \sum\limits_{i}\Big|\int\limits_{\R^d}\Big(\Phi(K_1*(f_{i} + f_{i+1}))-\Phi(K_1*f_i)-\Phi(K_1*f_{i+1})\Big)\Big|,
}
because the function
\eq{
\Phi(K_1*(f_{i} + f_{i+1}))-\Phi(K_1*f_i)-\Phi(K_1*f_{i+1})
}
vanishes outside~$L_i$. The first sum in~\eqref{LongFormula} is bounded by~\eqref{FromIndHyp}. For the second sum, we use the triangle inequality:
\mlt{
\sum\limits_{i}\Big|\int\limits_{\R^d}\Big(\Phi(K_1*(f_{i} + f_{i+1}))-\Phi(K_1*f_i)-\Phi(K_1*f_{i+1})\Big)\Big|\\ \leq \sum\limits_{i} \Big|\int\limits_{\R^d}\Phi(K_1*(f_{i} + f_{i+1}))\Big| + 2\Big|\int\limits_{\R^d}\Phi(K_1*f_{i})\Big|.
}
Since the functions~$f_i + f_{i+1}$ are supported on the parallelepipeds
\eq{\label{DimensionsOfPar}
\underbrace{\infty\times\infty\times\ldots\times\infty}_{D \text{ times}}\times 3^{d-D}\times 3^{d-D}\times\ldots \times 3^{d-D},
}  
we may apply the inductive hypothesis to them and complete the induction step.
\begin{figure}
	\includegraphics[width = 0.5\textwidth]{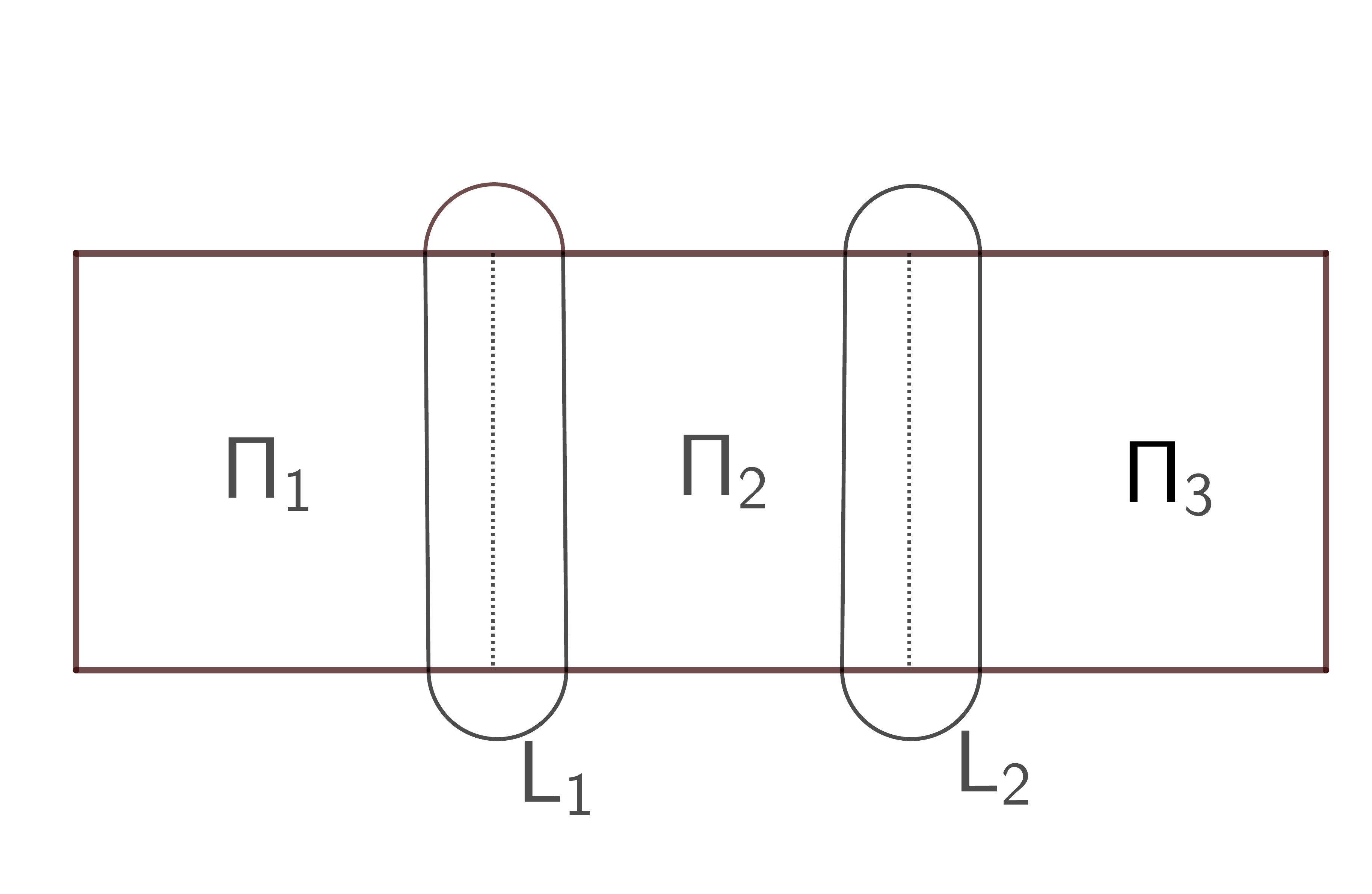}
	\includegraphics[width = 0.5\textwidth]{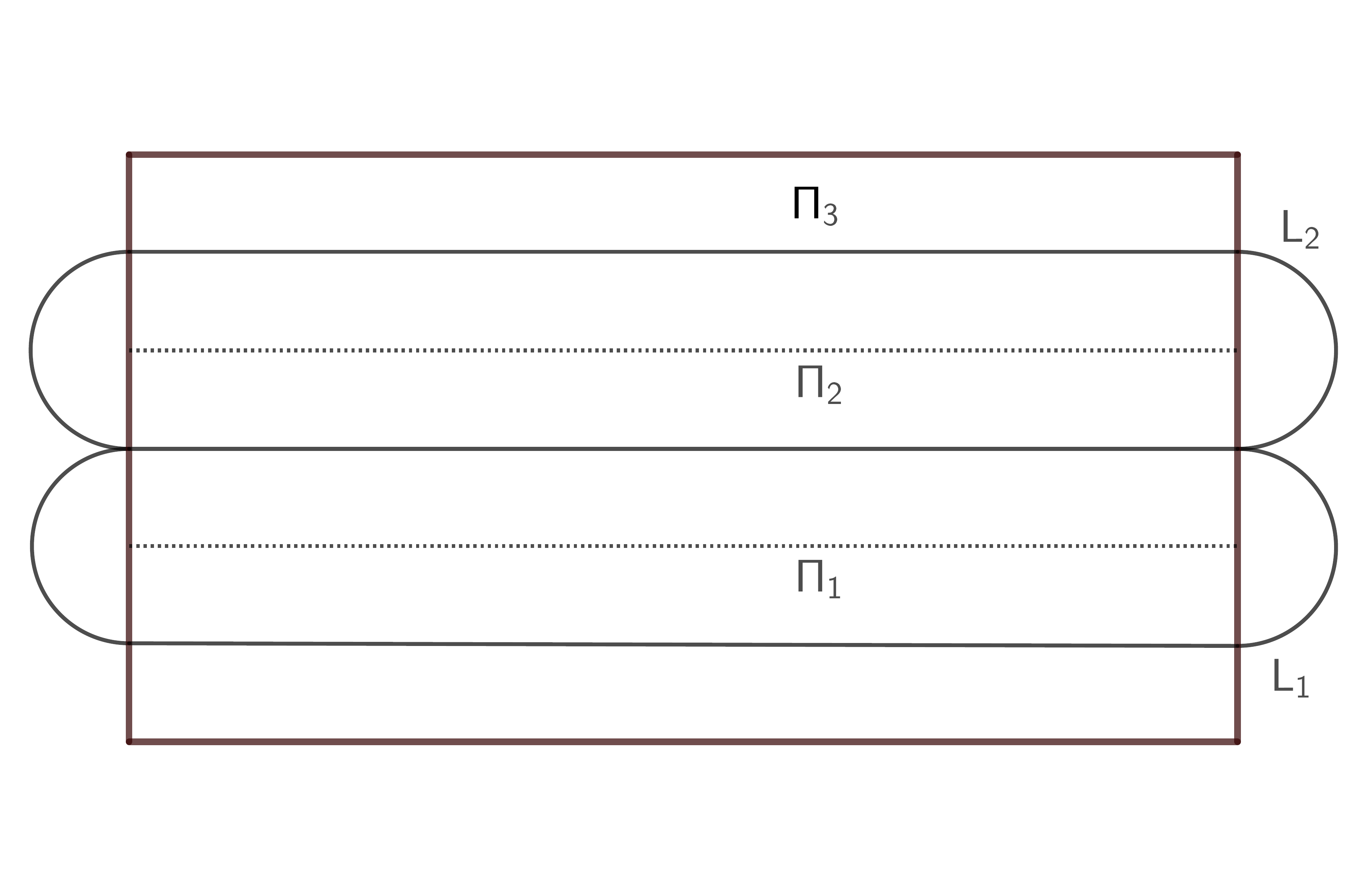}
	\caption{The rectangles~$\Pi_j$ and the sets~$L_j$ when~$d=2$ and~$D=0$ or~$D=1$. The common boundary of~$\Pi_j$ and~$\Pi_{j+1}$ is marked with a dotted line.}
	\label{SplittingIntoParal}
\end{figure}
\end{proof}
\begin{proof}[Proof of Theorem~\ref{Main2} in the case~$p=2$.]
In the light of Theorem~\ref{LocalMain2}, it suffices to prove the inequality
\eq{\label{Monulus}
\sum\limits_{n \geq 0}2^n \sum\limits_{j\in\Z^d}\|f\|_{L_1(3^d Q_{n,j})}^{}\inf\limits_{c_j}\int\limits_{3^d Q_{n,j}}|x-c_j||f(x)|\,dx\lesssim \|f\|_{L_1}^2.
}
We use the elementary estimate
\eq{
\|f\|_{L_1(3^d Q_{n,j})}^{}\inf\limits_{c_j}\int\limits_{3^d Q_{n,j}}|x-c_j||f(x)|\,dx \leq\!\!\!\!\!\! \iint\limits_{3^dQ_{n,j}\times 3^dQ_{n,j}}\!\!\!\!\!\! |x-y||f(x)||f(y)|\,dx\,dy,
}
and bound the left hand side of~\eqref{Monulus} by
\eq{
\sum\limits_{n\geq 0}\quad 2^n\!\!\!\!\!\!\!\!\iint\limits_{|x-y| \leq \sqrt{d}\, 3^d 2^{-n}}\!\!\!\! |x-y||f(x)||f(y)|\,dx\,dy.
}
It remains to notice that the functional series
\eq{
g(z) = \sum\limits_{n\geq 0}2^n|z|\chi_{B_{\sqrt{d}\, 3^d2^{-n}}(0)}(z)
}
is uniformly bounded (this estimate is similar to~\eqref{GeomSer}).
\end{proof}

\section{Martingale techniques}\label{S3}
If~$Q\subset \R^d$ is a cube, we denote its sidelength by~$\ell(Q)$ and the set of all its dyadic subcubes by~$\D(Q)$. This set admits a natural tree structure. Let~$\D_n(Q)$ be the set of all cubes~$Q'\in \D(Q)$ of generation~$n$; the cube~$Q$ itself is of generation~$0$.
\begin{Def}
Let~$f$ be a locally summable function\textup, let~$Q\subset \R^d$ be a cube\textup, let~$p \in [1,\infty)$. Define the quantity
\eq{
E_{Q,n}[f] = \sum\limits_{Q' \in \D_n(Q)}\Big(\int\limits_{Q'}|f(x)|\,dx\Big)^p. 
}
\end{Def}
One may see that the quantities~$E_{Q,n}[f]$ do not increase with~$n$. Thus,
\eq{\label{Telescope}
\|f\|_{L_1(Q)}^p = \sum\limits_{n\geq 0}\Big(E_{Q,n}[f] - E_{Q,n+1}[f]\Big),
}
and each summand in this telescopic sum is non-negative. 
\begin{Le}\label{EstimateByEnergyIncrement}
Let~$p \in (1,\infty)$. There exists~$\eps \in (0,1)$ such that
\eq{
 \|f\|_{L_1(Q)}^{p-1}\inf\limits_{c}\int\limits_{Q}\frac{|x-c|}{\ell(Q)} |f(x)|\,dx \lesssim \sum\limits_{n\geq 0}(1-\eps)^n\Big(E_{Q,n}[f] - E_{Q,n+1}[f]\Big)
 }
for any cube~$Q \subset \R^d$ and any summable~$f$ with compact support.
The constant in this inequality depends neither on~$f$ nor on the choice of~$Q$.
\end{Le}
\begin{Rem}
As we will see in the proof\textup, one may take~$\eps = \frac12-$. The quantity~$\ell(Q)$ in the denominator is a scaling factor related to the power~$2^n$ on the right hand side of~\eqref{MedianCorIneq}.
\end{Rem}
\begin{proof}
Without loss of generality, we may assume~$Q = [0,1]^{d}$. We will define the cubes~$R_0,R_1,\ldots,R_n,\ldots$ inductively. For each~$n$, we will have~$R_n \in \D_n(Q)$ and~$R_{n+1} \subset R_n$. In fact, we will always choose a cube with the largest mass (if there are several cubes with the largest mass, we may choose any one of them):
\eq{
\int\limits_{R_{n+1}}|f(x)|\,dx = \max\limits_{R \in \D_1(R_n)} \int\limits_{R}|f(x)|\,dx.
}
Let~$c_0$ be the unique common point of the cubes~$R_n$. We will prove the inequality
\eq{
 \|f\|_{L_1(Q)}^{p-1}\int\limits_{Q} |x-c_0||f(x)|\,dx \lesssim \sum\limits_{n\geq 0}(1-\eps)^n\Big(E_{R_n,0}[f] - E_{R_n,1}[f]\Big),
}
which implies the desired one.

Pick~$\delta > 0$ such that~$2(1-\delta)^{p-1} > 1$. Let~$N \in \N\cup \{0,\infty\}$ be the smallest possible~$n$ such that
\eq{
\int\limits_{R_{n+1}}|f(x)|\,dx \leq (1-\delta)\int\limits_{R_n}|f(x)|\,dx.
}
In particular,
\eq{
\|f\|_{L_1(Q)} \leq (1-\delta)^{-n} \|f\|_{L_1(R_n)},\quad n\leq N.
}
Then,
\mlt{
\|f\|_{L_1(Q)}^{p-1}\int\limits_{Q} |x-c_0||f(x)|\,dx\\ \lesssim \sum\limits_{n=0}^N \|f\|_{L_1(Q)}^{p-1} 2^{-n}\!\!\!\! \int\limits_{R_{n}\setminus R_{n+1}} |f(x)|\,dx + \|f\|_{L_1(Q)}^{p-1} 2^{-N} \int\limits_{R_{N+1}} |f(x)|\,dx\\ \lesssim
\sum\limits_{n=0}^N (2(1-\delta)^{p-1})^{-n}\|f\|_{L_1(R_n)}^{p-1} \int\limits_{R_{n}\setminus R_{n+1}} |f(x)|\,dx\\ + \delta^{-1}(2(1-\delta)^{p-1})^{-N} \|f\|_{L_1(R_N)}^{p-1}\int\limits_{R_{N}\setminus R_{N+1}} |f(x)|\,dx.
}
By Lemma~\ref{EnergyBound} below,
\eq{
\|f\|_{L_1(R_n)}^{p-1} \int\limits_{R_{n}\setminus R_{n+1}} |f(x)|\,dx \lesssim E_{R_n,0}[f] - E_{R_n,1}[f]
}
for any~$n \geq 0$. Thus,
\eq{
\|f\|_{L_1(Q)}^{p-1}\int\limits_{Q} |x-c_0||f(x)|\,dx \lesssim \sum\limits_{n=0}^N (2(1-\delta)^{p-1})^{-n}\Big(E_{R_n,0}[f] - E_{R_n,1}[f]\Big),
} 
so, the lemma holds true with~$\eps =1- (2(1-\delta)^{p-1})^{-1}$.
\end{proof}
We are almost ready to prove Theorem~\ref{Main2}. By Theorem~\ref{LocalMain2}, it suffices to verify the inequality
\eq{
\sum\limits_{n \geq 0}2^n \sum\limits_{j\in\Z^d}\|f\|_{L_1(3^d Q_{n,j})}^{p-1}\inf\limits_{c_j}\int\limits_{3^d Q_{n,j}}|x-c_j||f(x)|\,dx \lesssim \|f\|_{L_1}^p.
}
If we had the cubes~$Q_{n,j}$ instead of~$3^d Q_{n,j}$ on the left hand side, we could have used Lemma~\ref{EstimateByEnergyIncrement} and have reduced this inequality to~\eqref{Telescope} by interchanging the order of summation and the use of~$\sum_{n \geq 0}(1-\eps)^n \lesssim 1$. The problem is that the cubes~$3^dQ_{n,j}$ are not dyadic anymore, they do not form a tree. We will use the standard instrument for getting around this difficulty called the Three Lattice Theorem (see Theorem~$3.1$ in~\cite{LernerNazarov2019}). It says (in a slightly simplified form) that for any cube~$Q$ there exist cubes~$Q^{1},Q^{2},\ldots, Q^{s}$ with~$s = 3^d$ such that for any cube~$R\in \D(Q)$ there exists a number~$i =1,2,\ldots,s$ such that~$3R \in \D(Q^i)$. Of course, in such a case, the cubes~$Q^{i}$ may be taken of the size comparable to~$Q$. We may also iterate the Three Lattice Theorem and (at the cost of increment of~$s$) find new cubes~$Q^1,Q^2,\ldots, Q^s$ such that for any~$R\in \D(Q)$ there exists~$i =1,2,\ldots,s$ such that~$3^{d}R \in \D(Q^i)$. We apply the iterated Three Lattice Theorem to~$Q = [0,1]$ and obtain some cubes~$Q^1, Q^2,\ldots, Q^s$ such that for any~$Q_{n,j} \in \D([0,1]^d)$ there exists~$i = 1,2,\ldots, s$ such that~$3^d Q_{n,j}\in\D(Q^i)$.
\begin{proof}[Proof of Theorem~\ref{Main2}]
We apply Theorem~\ref{LocalMain2} and Lemma~\ref{EstimateByEnergyIncrement}:
\mlt{
\sum\limits_{n \geq 1}\Big|\int\limits_{\R^d} \Phi(K_n*f)\Big| \lesssim  \sum\limits_{n\geq 1}\sum\limits_{i=1}^s\sum\limits_{R\in \D_n(Q^i)} \sum\limits_{k\geq 0}(1-\eps)^k\Big(E_{R,k}[f] - E_{R,k+1}[f]\Big)\\ =
\sum\limits_{i=1}^s \sum\limits_{n \geq 1}\sum\limits_{k \geq 0}(1-\eps)^k\Big(E_{Q^i, n+k}[f] - E_{Q^i,n+k+1}[f]\Big)\\ =
\sum\limits_{i=1}^s\sum\limits_{m \geq 1}\bigg(\sum\limits_{k=0}^m(1-\eps)^k\bigg) \Big(E_{Q^i, m}[f] - E_{Q^i,m+1}[f]\Big)\\ \lesssim 
\sum\limits_{i=1}^s \sum\limits_{m \geq 1} \Big(E_{Q^i, m}[f] - E_{Q^i,m+1}[f]\Big) \Lseqref{Telescope} \|f\|_{L_1}^p.
}
\end{proof}
Though the proof of Theorem~\ref{Remainder} is lengthier than the proof above, it relies upon the same circle of ideas.
\begin{proof}[Proof of Theorem~\ref{Remainder}.]
We pick some large natural number~$N$ and use Lemma~\ref{Subadditive} below to split the sum on the left hand side
\mlt{\label{Subadd}
\int\limits_{\R^d}\Mp\Big(\big|K_{\leq n}*f(x)\big|, \big|K_{n+1}*f(x)\big|\Big)\,dx\\ \lesssim  \underbrace{\int_{\R^d}\Mp\Big(\big|K_{n-N\leq n}*f(x)\big|, \big|K_{n+1}*f(x)\big|\Big)\,dx}_{I_0^n} + \sum\limits_{m < n-N} \underbrace{\int\limits_{\R^d}\Mp\Big(\big|K_{m}*f(x)\big|, \big|K_{n+1}*f(x)\big|\Big)\,dx}_{I_m^n},
}
here we use the notation 
\eq{
K_{n-N\leq n} = \sum\limits_{k=n-N}^{n} K_k.
}
We will estimate the terms~$I_0^n$ and~$I_{m}^n$,~$m < n-N$, in slightly different ways.

\paragraph{Estimate of close terms.} We wish to prove the estimate
\eq{\label{LocalI0}
\int\limits_{Q_{n,j}}\Mp\Big(\big|K_{n-N\leq n}*f(x)\big|, \big|K_{n+1}*f(x)\big|\Big)\,dx \lesssim 2^n\|f\|_{L_1(3^dQ_{n,j})}^{p-1}\inf\limits_{c\in\R^d}\int\limits_{3^dQ_{n,j}} |x-c||f(x)|\,dx.
}
Once this is done, the desired inequality
\eq{
\sum\limits_{n\geq 0}I_0^n \lesssim \|f\|_{L_1}^p
}
will follow by application of the Three Lattice Theorem and Lemma~\ref{EstimateByEnergyIncrement} in the same way as in the proof of Theorem~\ref{Main2}.

Let us prove~\eqref{LocalI0}. Without loss of generality, we may assume~$n=0$ and~$\|f\|_{L_1(3^dQ_{0,j})} = 1$. Then, the quantities
\eq{
K_{-N\leq 0}*f(x)\quad \text{and}\quad K_{1}*f(x)
}
are bounded by~$O(1)$ for any~$x \in Q_{0,j}$. Thus, we may pick some~$c\in \R^d$ and rely upon the Lipschitz property of the function~$\Mp$ (Lemma~\ref{MpLip} below):
\mlt{
\int\limits_{Q_{0,j}}\Mp\Big(\big|K_{-N\leq 0}*f(x)\big|, |K_{1}*f(x)|\Big)\,dx \lesssim \int\limits_{Q_{0,j}}\Big|K_{-N\leq 0}*f(x) - K_{-N\leq 0}(x-c) \cdot \!\!\!\!\!\int\limits_{3^dQ_{0,j}}\!\!\!f\Big|\,dx\\ +
\int\limits_{Q_{0,j}}\Big|K_{1}*f(x) - K_{1}(x-c) \cdot\!\!\!\!\! \int\limits_{3^dQ_{0,j}}\!\!\! f\Big|\,dx + \int\limits_{Q_{0,j}}\Mp\Big(\big|K_{-N\leq 0}(x-c)\cdot\!\!\!\!\! \int\limits_{3^dQ_{0,j}}\!\!\!f\big|,\ \big|K_{1}(x-c)\cdot\!\!\!\!\! \int\limits_{3^dQ_{0,j}}\!\!\! f\big|\Big)\,dx.
}
The last term vanishes since the kernels~$K_{-N\leq 0}$ and~$K_1$ have disjoint supports. As for the first two terms, we may write the same estimate as~\eqref{LipschitzConvolution}, optimize in~$c$, and obtain~\eqref{LocalI0}.

\paragraph{Estimate for separated terms.} 
We will prove the estimate
\eq{\label{Locall1}
\int\limits_{Q_{m+N,j}}\Mp\Big(\big|K_{m}*f(x)\big|, \big|K_{n+1}*f(x)\big|\Big)\,dx \lesssim (1-\eps)^{n-m} \Big[E_{3^{d}Q_{m+N,j},0}[f] -E_{3^{d}Q_{m+N,j}, M}[f]\Big],
}
where~$\eps \in (0,1)$ and~$M\in \N$ are fixed numbers whose choice does not depend on~$f$,~$m$,~$n$, or~$j$. Recall~$n-m > N$. Similar to the proof of Theorem~\ref{Main2}, this inequality implies
\eq{
\sum\limits_{n\geq 0}\sum\limits_{m < n-N} I_{m}^n \lesssim \|f\|_{L_1}^p.
}
Let~$Q_{n,p}\subset Q_{m+N,j}$ be some other dyadic cube, then
\eq{
\big|K_{n+1}*f(x)\big|\lesssim 2^{n(d-\alpha)}\|f\|_{L_1(2Q_{n,p})},\qquad x\in Q_{n,p}
}
The number~$N$ should be so large that for any~$x\in Q_{m+N,j}$, the support of the function~$K_m(\cdot-x)$ intersects neither the cube~$Q_{m+N,j}$ nor the cubes~$Q_{m+N,i}$ with~$|i-j| \leq 2\sqrt d$.  Recall that the function~$K_m$ vanishes inside the ball~$B_{2^{-m-1}}(0)$, so, this requirement may be obtained by choosing sufficiently large~$N$ ($N=5d+5$ is fine). See Figure~\ref{Fig1} for visualization. Therefore,
\eq{
\big|K_{m}*f(x)\big| \lesssim 2^{m(d-\alpha)}\!\!\!\!\sum\limits_{2\sqrt{d} \leq |i-j| \lesssim 1}\|f\|_{L_1(Q_{m+N,i})},\qquad x\in Q_{n,p}.
}
\begin{figure}
\includegraphics[width=0.8\textwidth]{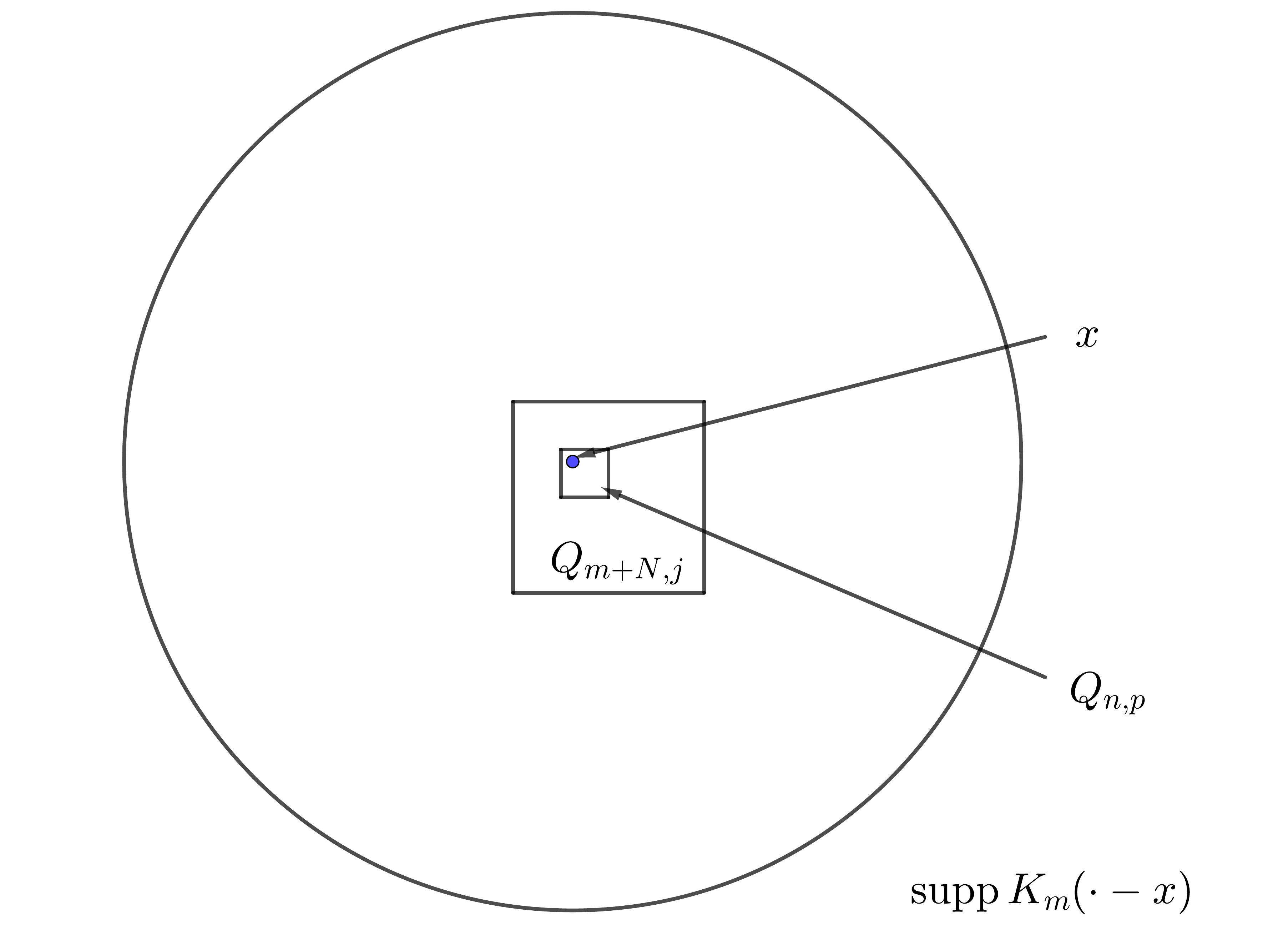}
\caption{The function~$K_m(\cdot - x)$ is supported outside~$Q_{m+N,j}$ and its neighbor cubes.}
\label{Fig1}
\end{figure}
Therefore,
\mlt{
\int\limits_{Q_{n,p}}\Mp\Big(\big|K_{m}*f(x)\big|,\big|K_{n+1}*f(x)\big|\Big)\,dx\\ 
\lesssim 2^{-nd}\Mp\Big(2^{m(d-\alpha)}\!\!\!\!\sum\limits_{2\sqrt{d} \leq |i-j| \lesssim 1}\|f\|_{L_1(Q_{m+N,i})},2^{n(d-\alpha)}\|f\|_{L_1(2Q_{n,p})}\Big).
}
We sum these estimates over all cubes~$Q_{n,p} \subset Q_{m+N,j}$ and use the concavity of the function~$y\mapsto \Mp(x,y)$ provided by Lemma~\ref{Convexity} below:
\mlt{\label{UnifiedCubes}
\int\limits_{Q_{m+N,j}}\Mp\Big(\big|K_{m}*f(x)\big|,\big|K_{n+1}*f(x)\big|\Big)\,dx\\ 
\lesssim 2^{-md}\cdot2^{(m-n)d}\!\!\!\!\!\!\sum\limits_{Q_{n,p}\subset Q_{m+N,j}}\!\!\!\!\!\Mp\Big(2^{m(d-\alpha)}\!\!\!\!\sum\limits_{2\sqrt{d} \leq |i-j| \lesssim 1}\|f\|_{L_1(Q_{m+N,i})},2^{n(d-\alpha)}\|f\|_{L_1(2Q_{n,p})}\Big)\\
\lesssim 2^{-md}\Mp\bigg(2^{m(d-\alpha)}\!\!\!\!\sum\limits_{2\sqrt{d} \leq |i-j| \lesssim 1}\|f\|_{L_1(Q_{m+N,i})},2^{(m-n)d+n(d-\alpha)}\!\!\!\!\!\sum\limits_{Q_{n,p}\subset Q_{m+N,j}}\|f\|_{L_1(2Q_{n,p})}\bigg)\\
\leq 2^{-md}\Mp\bigg(2^{m(d-\alpha)}\!\!\!\!\sum\limits_{2\sqrt{d} \leq |i-j| \lesssim 1}\|f\|_{L_1(Q_{m+N,i})},2^{(m-n)d+n(d-\alpha)}\|f\|_{L_1(2Q_{m+N,j})}\bigg).
}
We use the positive~$p$-homogeneity of the function~$\Mp$ and the identity~$d=(d-\alpha)p$ to rewrite the latter expression as
\eq{
\Mp\bigg(\sum\limits_{2\sqrt{d} \leq |i-j| \lesssim 1}\|f\|_{L_1(Q_{m+N,i})},2^{(m-n)\alpha}\|f\|_{L_1(2Q_{m+N,j})}\bigg).
}
It remains to notice that~$\Mp(x,\lambda y) \leq \lambda^{p-1}\Mp(x,y)$ when~$\lambda < 1$, and, thus, we finally arrive at
\mlt{
\int\limits_{Q_{m+N,j}}\Mp\Big(\big|K_{m}*f(x)\big|,\big|K_{n+1}*f(x)\big|\Big)\,dx\\
 \lesssim 2^{(m-n)(p-1)\alpha}\Mp\bigg(\sum\limits_{2\sqrt{d} \leq |i-j| \lesssim 1}\|f\|_{L_1(Q_{m+N,i})},\|f\|_{L_1(2Q_{m+N,j})}\bigg).
}
Therefore,~\eqref{Locall1} will be proved if we establish
\eq{\label{Wanted}
\Mp\bigg(\sum\limits_{2\sqrt{d} \leq |i-j| \lesssim 1}\|f\|_{L_1(Q_{m+N,i})},\|f\|_{L_1(2Q_{m+N,j})}\bigg) \lesssim\Big[E_{3^{d}Q_{m+N,j},0}[f] -E_{3^{d}Q_{m+N,j}, M}[f]\Big].
}
We notice that the sets
\eq{
\bigcup\limits_{2\sqrt{d} \leq |i-j| \lesssim 1} Q_{m+N,i}\qquad \text{and} \qquad 2Q_{m+N,j} 
}
are separated, and therefore, are covered by two disjoint subfamilies of the dyadic cubes in~$\D_{M}(3^dQ_{m+N,j})$, provided~$M$ is sufficiently large.  Thus,~\eqref{Wanted} follows from Lemma~\ref{EnergyBound2} below.
\end{proof}  

\section{Discussion}\label{SD}
\paragraph{Relation to martingale problem.}
Though we have not used the notions of conditional expectation, martingale, or the Bellman function in the proof, these concepts suggest what inequalities should be written and how the functions and the operators should be decomposed. We briefly comment on the analogies between the proof of Theorem~\ref{Main} and the martingale reasonings in~\cite{Stolyarov2021bis}. 

The main engine that controls everything is the collection of the quantities~$E_{Q,n}[f]$ and the telescopic sum~\eqref{Telescope} (similar things also play the major role in~\cite{Stolyarov2020}). In fact, we have several processes generated by the cubes~$Q^i$ provided by the Three Lattice Theorem. More precisely, each cube~$Q^i$ (let us assume for simplicity that~$\ell(Q^i) = 1$, which is not a restriction) and the function~$f$ generates a martingale by the formula
\eq{
F_{n}(x) = 2^{nd}\int\limits_{Q_{n,j}}|f(y)|\,dy,\qquad x\in Q_{n,j}\in \D_n(Q^i).
}
Then, the quanitity~$E_{Q,n}[f]$ may be expressed as
\eq{
E_{Q,n}[f] = 2^{-nd(p-1)}\E F_n^p.
}
This quantity somehow represents the part~$z^p$ of the supersolution from~\cite{Stolyarov2021bis} (see Theorem~$3.1$ therein). The main property of the quantities~$E_{Q,n}[f]$ we use is that the process~$2^{-nd(p-1)}\E F_n^p$ is a supermartingale.  The quantity~$\min(|y|^{p-1}z, |y|z^{p-1})$ from the same theorem is replaced with the quantities
\eq{
\Mp\Big(\big|K_{\leq n}*f(x)\big|, \big|K_{n+1}*f(x)\big|\Big).
}
Though the analogies are indirect here, the reader may compare the proofs in the present text and the proofs in~\cite{Stolyarov2021bis} and see that for each elementary inequality for Euclidean objects there is a similar inequality in the martingale world. 

We note that there are some differences between the discrete and continuous worlds. For example, there is no cancellation condition imposed on the kernel~$\tilde{K}$ itself in our reasonings. In~\cite{Stolyarov2021bis}, there was such a condition (see Definition~$2.2$ in that paper); though the condition appeared naturally in~\cite{Stolyarov2021bis}, its necessity was not proved.

\paragraph{The case $p > 2$.} The corrections to the proof in this case are also suggested by~\cite{Stolyarov2021bis}. The main idea is that in this case the function
\eq{
(x,y) \mapsto x^{p-1}y+xy^{p-1}
}
is locally Lipschitz. We replace~$\Mp(x,y)$ with this simpler function in all our reasonings. The only place in the argument where some changes are needed is the proof of Theorem~\ref{Remainder}. The function~$x\mapsto \Mp(x,y)$ is not subadditive anymore, so, a direct substitute for~\eqref{Subadd} is not allowed. One gets around this by allowing a tiny exponential multiple for~$I_{m}^n$ with the help of Lemma~$4.4$ in~\cite{Stolyarov2021bis} (see also the proof of Lemma~$5.3$ therein). Another difficulty comes from the fact that the function~$y\mapsto \Mp(x,y)$ is not concave anymore. However, this may be overcame by an application of  H\"older's inequality instead of Jensen's inequality in~\eqref{UnifiedCubes} (see the proof of Lemma~$5.3$ in~\cite{Stolyarov2021bis} for a similar computation).

\paragraph{Anisotropic things.} It would be interesting to extend Theorem~\ref{Main} to anisotropic setting. There is some evidence that the phenomenon of Bourgain--Brezis inequalities is present in this larger generality (see~\cite{KislyakovMaximov2018},~\cite{KMS2015}, and~\cite{Stolyarov2021} for some results in this direction). It is unclear how to transfer our methods to the anisotropic setting, because it lacks dyadic structure (in particular, what is the way to formulate the Three Lattice Theorem in the anisotropic setting?).

Let~$a\in \R^d$ be a vector with positive coordinates such that~$\sum_j a_j = d$. Consider the kernel~$K_a$ defined by the rule
\eq{
K_a\Big(t^{a_1}\zeta_1,t^{a_2}\zeta_2,\ldots ,t^{a_d}\zeta_d\Big) = t^{\alpha - d}\tilde{K}(\zeta),\qquad t\in \R_+,\zeta \in S^{d-1}.
}
Let as usual~$p=d/(d-\alpha)$ and let~$\Phi$ be a locally Lipschitz positively~$p$-homogeneous function.
\begin{Conj}
The inequality
\eq{
\Big|\int\limits_{\R^d} \Phi(K_a*f)\Big| \lesssim \|f\|_{L_1}^p,\qquad f\in C_0^\infty(\R^d), \ \int f=0,
}
holds true with a uniform constant if an only if
\eq{
\int\limits_{S^{d-1}} \Big(\sum\limits_{j=1}^d a_j\zeta_j^2\Big)\Phi(\tilde{K}(\zeta))\,d\zeta = 0\qquad\text{and}\qquad \int\limits_{S^{d-1}} \Big(\sum\limits_{j=1}^d a_j\zeta_j^2\Big)\Phi(-\tilde{K}(\zeta))\,d\zeta = 0.
}
\end{Conj}

\section{Auxiliary inequalities}\label{SAI}
\begin{Le}\label{K1}
The inequality
\eq{
\Big| K_{\leq 0}(x-y) - K_{\leq 0}(x)\Big| \lesssim \frac{|y|}{|x|^{d-\alpha+1}},\quad |x| \geq 2, |y| \leq 1,
}
holds true with a uniform constant.
\end{Le}
\begin{proof}
Since~$|x| \geq 2$ and~$|x-y|\geq 1$, we have
\mlt{
\Big| K_{\leq 0}(x-y) - K_{\leq 0}(x)\Big| = \Big|K(x-y) - K(x)\Big|\\ = |x|^{\alpha - d}\Big|K(x/|x|-y/|x|) - K(x/|x|)\Big| \lesssim |x|^{\alpha - d}\frac{|y|}{|x|} = \frac{|y|}{|x|^{d-\alpha+1}}.
}
We have used that~$K$ is homogeneous of order~$\alpha - d$ and locally Lipschitz outside the origin (note that~$|y|/|x| \leq \frac12$ in our case).
\end{proof}
\begin{Le}\label{K2}
Let~$p=2$. The inequality
\eq{
|K_{\leq n}|* |K_{n+1}| (x) \lesssim 2^n|x|(1+2^n|x|)^{-\frac{d}{2}-1}.
}
holds true\textup; the constant is uniform with respect to~$x\in \R^d$ and~$n \in \Z$.
\end{Le}
\begin{proof}
This inequality is dilation invariant, therefore, it suffices to consider the case~$n=0$:
\eq{
|K_{\leq 0}|* |K_{1}| (x) \lesssim |x|(1+|x|)^{-\frac{d}{2}-1}.
}
Recall~$\alpha = d/2$ in our case. The estimate
\eq{
|K_{\leq 0}|* |K_{1}| (x) \lesssim |x|^{-\frac{d}{2}},\quad |x|\geq 1,
}
is a consequence of the fact that~$K_1$ is a bounded function supported in~$B_\frac12(0)$ and
\eq{
\Big|K_{\leq 0}(x)\Big| \lesssim |x|^{-\frac{d}{2}}.
}
The estimate
\eq{
|K_{\leq 0}|* |K_{1}| (x) \lesssim |x|,\quad |x|\leq 1,
}
may be derived from the fact that the function on the left hand side is Lipschitz (since~$|K_{\leq 0}|$ is uniformly bounded and~$|K_1|$ is a function of bounded variation) and the fact that
\eq{
|K_{\leq 0}|*|K_1|(0) = \int\limits|K_{\leq 0}(x)||K_1(-x)|\,dx = 0
}
since~$K_{\leq 0}$ and~$K_1$ are supported outside~$B_{\frac12}(0)$ and inside~$B_{\frac12}(0)$, respectively.
\end{proof}
\begin{Le}\label{K3}
For any~$z,y\in\R^d$\textup, the inequality
\eq{
\int\limits_{\R^d}\Big|K_0(x-z) - K_0(x-y)\Big|\,dx \lesssim |z-y|
}
holds true with uniform constants.
\end{Le}
\begin{Rem}
The kernel~$K_0$ has compact support. If it were Lipschitz\textup, then the lemma would be trivial. However\textup, the kernel~$K_0$ has jumps on the spheres~$|x| = 1$ and~$|x| = \frac12$\textup, and the proof gets slightly more complicated.
\end{Rem}
\begin{proof}[Proof of Lemma~\ref{K3}]
Without loss of generality, we may assume~$z=0$ and~$|y|\leq 1/10$. The kernel~$K_0$ is defined by cases in~\eqref{Kn}. If both~$x-y$ and $x$ fall under the same case, then the bound
\eq{
\Big|K_0(x) - K_0(x-y)\Big| \lesssim |y|
}
holds true since~$\tilde{K}$ is a locally Lipschitz function. If the choice of the points~$x-y$ and~$x$ in~\eqref{Kn} leads to different cases, then~$x$ is~$y$-close either to the unit sphere or to the sphere of radius~$\frac12$ centered at the origin. Thus,
\eq{
\int\limits_{\R^d}\Big|K_0(x) - K_0(x-y)\Big|\,dx \lesssim |y| +\!\!\!\!\! \int\limits_{||x| - \frac12| \leq |y|}\!\!\!\!\! dx + \!\!\!\!\!\int\limits_{||x| - 1| \leq |y|}\!\!\!\!\!dx \lesssim |y|.
}
\end{proof}
\begin{Le}\label{Phi1}
For any~$a, b \in \R^\ell$ such that~$|b| \leq 2|a|$\textup, we have
\eq{
\Big|\Phi(a+b) - \Phi(a)\Big| \lesssim |a|^{p-1}|b|.
}
\end{Le}
\begin{proof}
We use the positive~$p$-homogeneity and the local Lipschitz property of~$\Phi$:
\eq{
\Big|\Phi(a+b) - \Phi(a)\Big| = |a|^p\Big|\Phi(a/|a|+b/|a|) - \Phi(a/|a|)\Big| \lesssim |a|^p\cdot \frac{|b|}{|a|} = |a|^{p-1}|b|.
}
\end{proof}
\begin{Le}\label{EnergyBound}
Let~$n \in \N$ and~$p \in (1,\infty)$. There exists a constant~$C = C(p,n) > 0$ such that
\eq{
\Big(\sum\limits_{j=1}^nz_j\Big)^p - \sum\limits_{j=1}^nz_j^p \geq C \Big(\sum\limits_{j=1}^n z_j\Big)^{p-1}\min_{j\in [1 .. n ]}\sum\limits_{i\ne j}z_i.
}
holds true for any choice of non-negative numbers~$z_1,z_2,\ldots,z_n$.
\end{Le}
\begin{proof}
Without loss of generality, let~$\sum_j z_j=1$. Let also~$z_1$ be the maximal of the~$z_j$; in particular,~$z_1 \in [n^{-1},1]$. Then, the left hand side is bounded from below by
\eq{
1 - \sum\limits_{j=1}^nz_j^p \geq  1- z_1^p - (1-z_1)^p,
}
and we are left with proving
\eq{
1-z_1^p - (1-z_1)^p \gtrsim 1-z_1,\qquad z_1 \in [n^{-1},1].
}
This inequality is true since the left hand side is a concave function that has zeros at~$z_1=0$ and~$z_1=1$.
\end{proof}
\begin{Le}\label{EnergyBound2}
Let~$n\in\N$ and~$p\in (1,\infty)$. There exists a constant~$C = C(p,n) > 0$ such that for any proper subset~$A \in [1.. n]$\textup, the inequality
\eq{
\Big(\sum\limits_{j=1}^nz_j\Big)^p - \sum\limits_{j=1}^nz_j^p \geq C\Mp\Big(\sum\limits_{j\in A}z_j,\sum\limits_{j\notin A}z_j\Big)
}
holds true for any choice of non-negative numbers~$z_1,z_2,\ldots,z_n$.
\end{Le}
\begin{proof}
The proof of this lemma is similar to the proof of the previous lemma. Assume~$\sum z_j = 1$. We estimate the left hand side from below by
\eq{
1-\Big(\sum\limits_{j\in A}z_j\Big)^p -\Big(\sum\limits_{j\notin A}z_j\Big)^p,
}
and reduce the problem to
\eq{
1-z^p - (1-z)^p \gtrsim \Mp(z,1-z),\qquad z\in [0,1].
}
This inequality follows from Lemma~\ref{MpLip} below.
\end{proof}
\begin{Le}\label{MpLip}
The function~$\Mp$ is locally Lipschitz.
\end{Le}
This lemma is completely similar to Lemma~$4.2$ in~\cite{Stolyarov2021bis} (note that, however, the function~$\Mp$ has different domain in that paper).
\begin{Le}\label{Subadditive}
The function~$\Mp$ is subadditive in the following sense\textup: for any~$p \in (1,2)$ the inequality
\eq{
\Mp\big(\sum_n a_n,b\big) \lesssim \sum_n\Mp(a_n,b),\qquad b, a_n \in \R_+
} 
holds true.
\end{Le}
\begin{proof}
We will use the representation
\eq{
\Mp(x,y) = y^p\theta\Big(\frac{x}{y}\Big),\quad x,y\in\R_+,
}
where~$\theta(t) = \min(|t|,|t|^{p-1})$,~$t \in \R$. It suffices to prove the inequality
\eq{
\theta\big(\sum_n t_n\big) \lesssim \sum_n\theta(t_n)
}
for non-negative numbers~$t_n$. It follows from the inequality
\eq{
\big|\theta(a+b) - \theta(a)\big| \lesssim \theta(b)
}
proved in Lemma~$4.3$ in~\cite{Stolyarov2021bis}.
\end{proof}
\begin{Le}\label{Convexity}
Let~$p \in (1,2]$. The function~$y\mapsto \Mp(x,y)$ is concave on the half-line~$\R_+$ for any fixed~$x \in \R_+$.
\end{Le}
\begin{proof}
The lemma follows from the representation~$\Mp(x,y) = x^p\theta(y/x)$, where~$\theta(t) = \min(t,t^{p-1})$, and the concavity of~$\theta$.
\end{proof}

\bibliography{mybib}{}
\bibliographystyle{amsplain}

\bigskip

St. Petersburg State University, Department of Mathematics and Computer Science;


d.m.stolyarov at spbu dot ru.
\end{document}